\RequirePackage[l2tabu, orthodox]{nag}
\documentclass[11pt]{amsart} 

\author{Layne Hall}
\address{School of Mathematics and Statistics, The University of Melbourne, Victoria 3010 Australia}
\email{lrhall@student.unimelb.edu}
\author{Andy Hammerlindl}
\address{School of Mathematical Sciences, Monash University, Victoria 3800 Australia}
\urladdr{ http://users.monash.edu.au/~ahammerl/} 
\email{andy.hammerlindl@monash.edu}
\title{Classification of partially hyperbolic surface endomorphisms}
\newcommand\funding[1]{%
	\begingroup
	\renewcommand\thefootnote{}\footnote{#1}%
	\addtocounter{footnote}{-1}%
	\endgroup
}
\providecommand{\keyword}[1]
{\textbf{Keywords:} #1}

\usepackage{amssymb}
\usepackage{amsmath}
\usepackage{amsthm}
\usepackage{mathtools}
\usepackage{float}
\usepackage{bold-extra}
\usepackage[font={small}]{caption}
\usepackage{afterpage}
\usepackage{subcaption}
\usepackage{verbatim}
\usepackage[inline]{asymptote}
\usepackage{pgfplots}
\usepackage{csquotes}
\usepackage{bm}
\usepackage{enumerate}
\usepackage{textcomp}
\usepackage{titlesec}
\usepackage{nameref}
\usepackage{lipsum}
\usepackage{enumitem}
\usepackage{mathrsfs}
\usepackage[capitalise]{cleveref}
\theoremstyle{definition}

\theoremstyle{Theorem}
\newtheorem{thm}{Theorem}

\newtheorem{theorem}{Theorem}

\newtheorem{prop}[thm]{Proposition}
\newtheorem{sublem}[thm]{Sublemma}
\theoremstyle{Theorem}
\newtheorem{lem}[thm]{Lemma}
\theoremstyle{Theorem}

\theoremstyle{Theorem}
\newtheorem{cor}[thm]{Corollary}
\theoremstyle{remark}
\newtheorem{rem}[thm]{Remark}

\DeclareMathOperator{\eps}{\varepsilon}

\DeclareMathOperator{\bbT}{\mathbb{T}}
\DeclareMathOperator{\bbR}{\mathbb{R}}

\DeclareMathOperator{\bbZ}{\mathbb{Z}}
\DeclareMathOperator{\Cone}{\mathcal{C}}

\def\restrict#1{\raise-.5exhbox{\ensuremath|}_{#1}}
\DeclareMathOperator{\bran}{\mathcal{F}^c_{\mathrm{bran}}}

\DeclareMathOperator{\feps}{\mathcal{F}^c_{\mathrm{\eps}}}

\DeclareMathOperator{\dist}{\mathrm{dist}}
\DeclareMathOperator{\lf}{\mathcal{L}}

\DeclareMathOperator{\fmax}{\mathcal{F}_{\mathrm{max}}}
\DeclareMathOperator{\fmin}{\mathcal{F}_{\mathrm{min}}}
\DeclareMathOperator{\fcw}{\mathcal{F}_{\mathrm{cw}}}
\DeclareMathOperator{\fccw}{\mathcal{F}_{\mathrm{ccw}}}

\newcommand{\inv}{^{-1}}
\numberwithin{thm}{section}
\numberwithin{equation}{section}
\numberwithin{figure}{section}

\titleformat{\title}
{\normalfont\scshape\center}{\title}{1em}{}
\titleformat{\section}
{\normalfont\scshape\center}{\thesection}{1em}{}
\titleformat{\subsection}
{\normalfont\bfseries}{\thesubsection}{1 em}{}
\usepackage{fancyhdr}

\usepackage[
backend=biber,
style=alphabetic,
natbib=true,
url=false, 
doi=true,
eprint=false,
maxbibnames=9,
maxcitenames=9
]{biblatex}
\begin{document}
	\maketitle
	\begin{abstract}
	We show that in the absence of periodic centre annuli, a partially hyperbolic surface endomorphism is dynamically coherent and leaf conjugate to its linearisation. We proceed to characterise the dynamics in the presence of periodic centre annuli. This completes a classification of partially hyperbolic surface endomorphisms.
	\\
	
	\noindent \keyword{Partial hyperbolicity, Non-invertible dynamics, Dynamical coherence.}
	\end{abstract}
	\funding{This work was partially funded by the Australian Research Council.}
	\section{Introduction}
    The dynamics of non-invertible surface maps are less understood than their invertible counterparts. In \cite{enco}, it is shown that certain classes of partially hyperbolic surface endomorphisms are leaf conjugate to linear maps. Such a comparison to linear maps cannot be achieved in general, with \cite{enco}, \cite{heshiwang} and \cite{tan} all constructing examples which do not admit centre foliations. Further, \cite{tan} introduces the notion of a periodic centre annulus as a geometric mechanism for failure of integrability of the centre direction. In this paper, we show that periodic centre annuli are the unique obstruction to dynamical coherence, and give a classification up to leaf conjugacy in their absence. We further give a characterisation of endomorphisms with periodic centre annuli, completing a classification of partially hyperbolic surface endomorphisms.
	
	Before stating our results, we recall preliminary definitions. A \emph{cone family} $\Cone \subset TM$ consists of a closed convex cone $\Cone(p) \subset T_p M$ at each point $p \in M$. A cone family is \emph{$Df$-invariant} if $D_p f \left(\Cone(p)\right)$ is contained in the interior of $\Cone(f(p))$ for all $p\in M$. A map $f:M\to M$ is a \emph{(weakly) partially hyperbolic endomorphism} if it is a local diffeomorphism and it admits a cone family $\Cone^u$ which is $Df$-invariant and such that $1 < \| Df v^u \|$
	for all $v^u \in \Cone^u$. We call $\Cone^u$ an unstable cone family.  Let $M$ be a closed oriented surface. A partially hyperbolic endomorphism on $M$ necessarily admits a centre direction, that is, a $Df$-invariant line field $E^c\subset TM$ \cite[Section 2]{cropot2015lecture}. The existence of this line field and orientability in turn imply that $M=\bbT^2$.
	
    The homotopy classes of endomorphisms play a key role in their classification. To each partially hyperbolic endomorphism $f:\bbT^2\to\bbT^2$, there exists a unique linear endomorphism $A: \bbT^2 \to \bbT^2$ which is homotopic to $f$. We call $A$ the \emph{linearisation} of $f$. If $\lambda_1$ and $\lambda_2$ are the (not necessarily distinct) eigenvalues of $A$, then $A$ is one of three types:
	\begin{itemize}
		\item if $|\lambda_1| < 1 < |\lambda_2|$, we say $A$ is \emph{hyperbolic} if, 
		\item if $1<|\lambda_1| \leq |\lambda_2|$, we say $A$ is \emph{expanding}, and
		\item if $1=|\lambda_1|<|\lambda_2|$, we say $A$ is \emph{non-hyperbolic}.
	\end{itemize}
	When $A$ is hyperbolic or expanding, then we have the Franks semiconjugacy $H:\bbT^2\to\bbT^2$ from $f$ to $A$ \cite{Franks1}.
	
	The existing classification of surface endomorphisms relies on integrability of the centre direction. We say a partially hyperbolic endomorphism of $\bbT^2$ is \emph{dynamically coherent} if there exists an $f$-invariant foliation tangent to $E^c$. Otherwise, we say that the endomorphism is \emph{dynamically incoherent}. Though an endomorphism does not necessarily admit a centre foliation, it always admits a centre \emph{branching foliation}: an $f$-invariant collection of $C^1$ curves tangent to $E^c$ which cover $\bbT^2$ and do not topologically cross. A construction of such an object tangent to a general continuous distribution is carried out in \S 5 of \cite{BI}.
	
	Suppose that $f,\, g:\bbT^2 \to \bbT^2$ are partially hyperbolic endomorphisms which are dynamically coherent. We say that  $f$ and $g$ are \emph{leaf conjugate} if there exists a homeomorphism $h: \bbT^2 \to \bbT^2$ which takes centre leaves of $f$ to centre leaves of $g$, and which satisfies
	\[
	h(f(\mathcal{L})) = g(h(\mathcal{L}))
	\]
for every center leaf $\mathcal{L}$ of $f$.
	In \cite{enco}, it shown that if $f:\bbT^2\to\bbT^2$ is a partially hyperbolic surface endomorphism whose linearisation $A$ is hyperbolic, then $f$ is dynamically coherent and leaf conjugate to $A$.
	
    There also exist dynamically incoherent partially hyperbolic surface endomorphisms. The first known examples were constructed in \cite{heshiwang} and \cite{enco} and are homotopic to non-hyperbolic linear maps. Examples homotopic to expanding linear maps were later constructed in \cite{tan}, leading to the result that every linear map on $\bbT^2$ with integer eigenvalues $\lambda_1, \lambda_2$ satisfying $|\lambda_2| > 1$ is homotopic to an incoherent partially hyperbolic surface endomorphism. A geometric mechanism for incoherence called a periodic centre annulus was introduced in \cite{tan}. A \emph{periodic centre annulus} is an immersed open annulus $X\subset\bbT^2$ such that $f^k(X)=X$ for some $k>0$ and whose boundary, which must consist of either one or two disjoint circles, is $C^1$ and tangent to the centre direction. We further require that a periodic centre annulus is minimal, in the sense that there is no smaller annulus $Y\subsetneq X$ with the same properties.
	
	Of the incoherent examples constructed in \cite{tan}, an interesting class are those homotopic to linear maps which are homotheties or non-trivial Jordan blocks, which themselves as maps on $\bbT^2$ are not partially hyperbolic. Our first result is that such linearisations are, in a sense, defective. 
	\begin{theorem}\label{thm:linearisation}
		Let $f:\bbT^2\to\bbT^2$ be a partially hyperbolic endomorphism with linearisation $A:\bbT^2\to\bbT^2$. If $f$ does not admit a periodic centre annulus, then the eigenvalues of $A$ have distinct magnitudes.
	\end{theorem}
	In light of the preceding theorem, the absence of a periodic centre annulus implies that the linearisation is itself a partially hyperbolic endomorphism. In this absence it then makes sense to talk of leaf conjugacies of a map to its linearisation. A main result of this paper is that periodic centre annuli are the unique obstruction to classification up to leaf conjugacy.
	\begin{theorem}\label{thm:classi}
		Let $f:\bbT^2\to\bbT^2$ be partially hyperbolic. If $f$ does not admit a periodic centre annulus, then $f$ is dynamically coherent and leaf conjugate to its linearisation.
	\end{theorem}
	To give a comprehensive classification of partially hyperbolic surface endomorphisms, we look to understand dynamics of endomorphisms with periodic centre annuli. Our setting for this classification bears some similarity to diffeomorphisms on $\bbT^3$, where centre-stable tori are known to be the unique obstruction to both dynamical coherence and leaf conjugacy \cite{potrie,hp2014pointwise}. The analogue of a periodic centre annulus in this other setting is a region between centre-stable tori. A classification of such diffeomorphisms is given in \cite{clab}, where it is shown that there are finitely many centre-stable tori and the dynamics on the regions between such tori take the form of a skew product. The finiteness of periodic centre annuli also holds in our setting, and while not difficult to prove, is essential for our classification.
	\begin{theorem}\label{thm:finiteannuli}
		Let $f:\bbT^2\to\bbT^2$ be a partially hyperbolic endomorphism. Then $f$ admits at most finitely many periodic centre annuli $X_1,\dots,X_n$.
	\end{theorem}
	We can follow the same procedure as \cite{clab} to show that the dynamics on a periodic centre annulus also takes the form of a skew product. This result is stated as \cref{prop:skew} and discussed in \cref{sec:orbits}, though as it is a straightforward adaption, its complete proof is deferred to an auxiliary document \cite{clab2}.
	
	Unlike diffeomorphisms on $\bbT^3$, the fact that the dynamics on a periodic centre annulus takes the form of a skew product is not a complete classification in our setting, as it is not necessarily true that the annuli cover $\bbT^2$. This is seen in the examples of \cite{tan}. However, this is at least true for the examples in \cite{heshiwang} and \cite{enco}, and in fact holds for any endomorphism with non-hyperbolic linearisation, as stated in the following result.
	\begin{theorem}\label{thm:nonhypclassi}
        Let $f:\bbT^2\to\bbT^2$ be a partially hyperbolic endomorphism with
        non-hyperbolic linearization $A$
        and with at least one periodic center annulus.
        Then $\bbT^2$ is the union of the closures
        of all of the periodic center annuli of $f$.
	\end{theorem}
	
	It then remains to characterise the dynamics when the linearisation is expanding, where the annuli the collection of closed periodic centre annuli will not cover $\bbT^2$. This is our final result, completing a classification of partially hyperbolic surface endomorphisms.
	\begin{theorem}\label{thm:expandclassi}
		Let $f:\bbT^2\to\bbT^2$ be a partially hyperbolic endomorphism which admits a periodic centre annulus, and let $X_1,...,\,X_n$ be the collection of all disjoint periodic centre annuli of $f$. Let $A:\bbT^2\to\bbT^2$ be the linearisation of $f$ and suppose $A$ is an expanding linear map, so that there is the Franks semiconjugacy $H:\bbT^2\to\bbT^2$ from $f$ to $A$. Then:
		\begin{itemize}
			\item The set
			\[
			\Lambda = \bigcup_i \bigcup_{k\geq 0}f^{-k}(X_i)
			\]
			is dense in $\bbT^2.$ If $X$ is a connected component in this set, then $X$ is an annulus, and $H(X)$ is either a periodic or preperiodic circle under $A$.
			\item The complement
			\[
			\bbT^2\setminus\, \Lambda
			\]
			is the union of disjoint circles tangent to the center direction of $f$. If $S$ is a connected component of $\Lambda$, then $S$ is a circle, and is either periodic/preperiodic under $f$, or $H(S)$ is a circle which is transitive under $A$.
		\end{itemize}
	\end{theorem}
	This result can be phrased as saying that either a point eventually lies in a periodic centre annulus, where we understand the dynamics as a skew product due to our earlier discussion, or it lies on an exceptional set of center circles
    where the dynamics can be understood using the semiconjugacy.
	
	The current paper is structured as follows. We begin in \cref{sec:inside} by studying how periodic centre annuli manifest within centre branching foliations. In particular, a Reeb component or what we call a tannulus give rise to periodic centre annuli. This characterisation is fundamental to the rest of the paper, and in this process we prove \cref{thm:finiteannuli}.
	
	We next establish \cref{thm:linearisation} in \cref{sec:para}. This is done by introducing rays associated to branching foliations as a tool for relating the dynamics of $f$ to its linearisation $A$, an idea that is used only within that section.
	
	\cref{sec:coherent} contains the proof of \cref{thm:classi}. We use topological properties of the branching foliation that were established in the preceding sections, paired with the Poincar\'e-Bendixson Theorem, to establish dynamical coherence. A leaf conjugacy is then constructed with an averaging technique, similar to \cite{enco}. 
	
	We conclude the paper by proving both Theorems \ref{thm:nonhypclassi} and \ref{thm:expandclassi} in \cref{sec:orbits}.

	\section{Tannuli and Reeb components}\label{sec:inside}
	In this section, we characterise how periodic centre annuli can arise in branching foliations. These are fundamental to proving the main theorems of the paper, and we shall prove \cref{thm:finiteannuli} in the process.
	
	Let $f:\bbT^2\to\bbT^2$ be a partially hyperbolic surface endomorphisms.  Lift $f$ to a diffeomorphism $\tilde{f}:\bbR^2\to\bbR^2$, recalling that $\tilde{f}$ admits an invariant splitting $E^u \oplus E^c$ \cite{manepughendo}. For an unstable segment $J^u\subset \bbR^2$, define $U_1(J) = \{p\in\bbR^2: \mathrm{dist}(p,J)<1\}$. The following result, which is the basis of `length vs.~ volume' arguments, is proved for the endomorphism setting in Section 2 of \cite{enco}.
	\begin{prop}\label{lengthvolume}
		There is $K>0$ such that if $J \subset \bbR^2$ is either an unstable segment or a centre segment, then
		\[
		\mathrm{volume}(U_1(J))> K \, \mathrm{length}(J).
		\] 
	\end{prop}
	Recall that an essential circle $\mathcal{C}\subset\bbT^2$ has a slope.
    This can be defined in terms of the first homology group of $\bbT^2$
	\begin{lem}
		Let $X_1, X_2 \subset \bbT^2$ be periodic centre annuli of $f$. Then the boundary circles of $X_1$ and $X_2$ have the same slope. Moreover, these circles do not topologically cross.
	\end{lem}
	\begin{proof}
		Let $X_1$ and $X_2$ be two periodic centre annuli. By replacing $f$ with an iterate, we may assume that $X_1$ and $X_2$ are $f$-invariant. 
		
		First suppose that $X_1$ and $X_2$ have different slopes.
        Let $\pi : \bbR^2 \to \bbT^2$ be the universal covering map.
        Since the slopes are different, we can find a lift $f$
        to a map $\tilde f : \bbR^2 \to \bbR^2$ and connected components
        $\tilde{X_1}$ of $\pi \inv (X_1)$ and
        $\tilde{X_2}$ of $\pi \inv (X_2)$ 
        such that $\tilde f(\tilde{X_1}) = \tilde{X_1}$,
        and $\tilde f(\tilde{X_2}) = \tilde{X_2}$.
        Then $\tilde{X}_1$ and $\tilde{X}_2$ each lie a bounded distance from lines $L_1, L_2\subset \bbR^2$ with different slopes. Then $\tilde{X}_1\cap \tilde{X}_2$ lies inside finite parallelogram bounded by edges parallel to $L_1$ and $L_2$. Thus $\tilde{X}_1\cap \tilde{X}_2$ has bounded volume. But this region is invariant, and so a small unstable curve inside this region grows exponentially in length. By \cref{lengthvolume}, this is a contradiction.
        
        Now assume that the boundary circles of the annuli have the
        same slope and that a boundary circle of $X_1$ topologically
        crosses a boundary circle of $X_2$.
        Let $\tilde{f}$ be a lift of $f$
        and $\tilde{X_1}$ be a connected component of $\pi \inv(X_1)$
        such that $\tilde{f}(\tilde{X_1}) = \tilde{X_1}$.
        By lifting a point $x \in X_1 \cap X_2$, we can find a point
        $\tilde{x} \in \tilde{X_1}$ whose orbit under $\tilde{f}$
        lies in $\tilde{X_1} \cap \pi \inv (X_2)$ for all time.
        Since only finitely many connected components of 
        $\pi \inv (X_2)$ intersect $\tilde{X_1}$
        this means that there is a connected component
        $\tilde{X_2}$ of $\pi \inv (X_2)$ 
        which is periodic under $\tilde{f}$.
        Replacing $f$ by an iterate,
        we may assume that
        $\tilde f(\tilde{X_2}) = \tilde{X_2}$.

        Since the annuli $X_1$ and $X_2$
        are of the same slope, the $\tilde{f}$-invariant set $\tilde{X_1}\cap
        \tilde{X}_2$ lies within a bounded neighbourhood of some line $L$. By
        making a linear coordinate change on $\bbR^2$, we may assume that $L$
        is vertical, and that the lifts $\tilde{X}_1$ and $\tilde{X}_2$ are
        both invariant under vertical translation. Let $\mathcal{L}_1$ and
        $\mathcal{L}_2$ be lifts of boundary circles of $\tilde{X}_1$ and
        $\tilde{X}_2$, respectively, that cross. Denote the connected
        components of $\bbR^2\setminus\mathcal{L}_1$ as  $U^+$ and $U^-$. The
        existence of a topological crossing means there is a compact set $S$
        that is either a point or interval with $S \subset
        \mathcal{L}_1\cap\mathcal{L}_2$, such that $\mathcal{L}_2$ passes from
        $U^-$ about a neighbourhood of one endpoint of $S$, and then enters
        $U^+$ about a neighbourhood of the other endpoints of $S$. Since both
        $\mathcal{L}_1$ and $\mathcal{L}_2$ are invariant under vertical
        translation, then $\mathcal{L}_2$ also crosses from $U^-$ to $U^+$
        locally at $S+(1,0)$. By a connectedness argument, $\mathcal{L}_2$
        must then have crossed back from $U^+$ to $U^-$ at some point $q$ on
        $\mathcal{L}_1$ between $S$ and $S+(1,0)$. Since both $\mathcal{L}_1$
        and $\mathcal{L}_2$ are $C^1$, then there must be a disc bounded by a
        bigon consisting of a segment along one of each of $\mathcal{L}_1$ and
        $\mathcal{L}_2$. Both edges of the bigon lie somewhere along the lines
        between the crossings $S$ and $S+(0,1)$. We can also see that such a
        bigon whose vertices are at crossings of $\mathcal{L}_1$ and
        $\mathcal{L}_2$ must have uniformly bounded volume: It must be bounded
        by $1$ in the vertical direction, and since it lies a uniform distance
        from a vertical line $L$, it is also uniformly bounded in the
        horizontal direction. Since $\tilde{f}$ is a diffeomorphism, then the
        open disc $V$ is mapped again to a disc bounded by another one of
        these particular bigons. Thus $\tilde{f}^n(V)$ must have uniformly
        bounded volume for all $n$. By considering a small unstable curve
        $J^u\subset V$, then \cref{lengthvolume} gives a contradiction.
	\end{proof}
    Let $X \subset \bbT^2$ be an open annulus. We define the \emph{width} of $X$ be the Hausdorff distance between its boundary circles. The preceding lemma allows us to prove \cref{thm:finiteannuli}.
	\begin{proof}[Proof of \cref{thm:finiteannuli}]
		We use what is basically the idea of local product structure, but instead in the context of an unstable cone. Since $\Cone^u$ and $E^c$ are a bounded angle apart, then there is $\delta>0$ such that if a periodic centre annulus $X$ contains an unstable curve of length $1$, then $X$ must have width at least $\delta$. Since $X$ is periodic, then by iterating a small unstable curve within $X$, we see that $X$ must contain an unstable curve of length $1$. Thus any periodic centre annulus has width at least $\delta$. But by the preceding lemma, all periodic centre annuli are disjoint, so by compactness there can only be finitely many.
	\end{proof}
	Periodic centre annuli will be easier to understand if we can consider them as a part of branching foliation.
	\begin{lem}\label{lem:inbran}
		There exists an invariant centre branching foliation $\bran$ of $\bbT^2$ which contains the boundary circles of all periodic centre annuli.
	\end{lem}
	\begin{proof}
		Let $X\subset \bbT^2$ be a periodic centre annulus. Begin with any invariant centre branching foliation $\bran$ of $\bbT^2$.
        As mentioned in the introduction, the existence of such a branching
        foliation is given in Section 5 of \cite{BI}.
        Then we can construct $\bran$ by cutting a leaf at each point it crosses a boundary of circle of $X$ and completing these cut segments to leaves of $\bran$ by gluing the boundary circles of $X$ to each component. We also retain the leaves which do not cross the boundary circles of $X$, and obtain a branching foliation of $\bran$ which do not cross the boundary circles of $X$. We can then freely add these circles to $\bran$ to obtain a branching foliation which contains $X$. There are finitely many periodic centre annuli whose boundaries do not cross each other, and so we can repeat this process to add all periodic centre annuli to $\bran$.
	\end{proof}
	For the remainder of this section, we will take $\bran$ to be an invariant centre branching foliation contain all periodic centre annuli. Recall that a branching foliation also has an associated approximating foliation $\feps$ as in \cite{BI}. Then the slope of $\feps$ is necessarily rational, and so either the foliation is the suspension of some circle homeomorphism or contains a Reeb component. In the dynamically incoherent example of \cite{enco} and \cite{heshiwang}, the periodic centre annuli correspond to Reeb components in the approximating foliation. In the example of \cite{tan} and the dynamically coherent but not uniquely integrable example of \cite{heshiwang}, the centre curves form what we call an \emph{tannulus}, due to its similarities to a collection of translated graphs of the tangent function. Rigorously, a tannulus is a foliated closed annulus such that every leaf in the interior is homeomorphic to $\bbR$ and which when lifted to a foliation on a strip $I\times\bbR \subset\bbR^2$, is such that the leaves tend to one boundary component of $I\times \bbR$ with $y$-coordinate approaching $\infty$ when followed forwards, and then tending toward the other component with $y$-coordinate approaching $-\infty$ followed backwards. We say that an annulus in $\bran$ is a Reeb component or tannulus if its approximation in $\feps$ is such an annulus.
	
	Recall that if $A$ is the linearisation of $f$, then $\bran$ is almost parallel to an $A$-invariant linear foliation $\mathcal{A}$. Let $\pi:\bbR^2\to\bbR$ be a linear projection which maps each leaf of $\mathcal{A}$ onto $\bbR$.
	\begin{lem}
		The image and preimage of a Reeb component in $\bran$ under $f$ is a Reeb component. Similarly, the image and preimage of a tannulus in $\bran$ is a tannulus.
	\end{lem}
	\begin{proof}
		Lift $\bran$ to $\bbR^2$. Reeb components and tannuli in $\bran$ cannot map to annuli foliated by circles, so it remains to show that a Reeb component cannot map to tannulus and vice versa. Observe that if $\mathcal{L}$ is a leaf of $\bran$ contained in a Reeb, then $\pi(\lf)$ is an interval of the form $[a,\infty)$ for some finite constant $a$. however, if $\mathcal{L}'$ is contained within a tannulus, then $\pi(\lf')$ is all of $\bbR$. Then $\pi A(\lf)$ is a half-bounded interval, and $\pi A(\lf')$ is all of $\bbR$. Since $f$ is a finite distance from $A$, we cannot have $f(\lf) = \lf'$ or $f(\lf') = \lf$.
	\end{proof}
	We now show that the presence of Reeb components and tannuli necessitate the existence of a periodic centre annulus.
	\begin{lem}\label{stripsareperiodic}
		If $X$ is a tannulus or Reeb component of a branching foliation $\bran$, then there is $k\geq 0$ such that $f^k(X)$ is periodic.
	\end{lem}
	\begin{proof}
		If $\bran$ contains a Reeb component, then since $\bran$ can contain only finitely many Reeb components, there is a periodic Reeb component. If $\bran$ contains a tannulus, denote this tannulus as $X$. If $J^u \subset X$ is an unstable curve, then $f^k(J^u)$ grows unbounded in length, implying that $f^k(X)$ must have width bounded below by some $\delta>0$ for all sufficiently large $k$. However, it is clear from their definition that all tannuli can intersect only on their boundary. Thus there can only be finitely many tannuli of width $\delta$, so that $f^k(X)$ is periodic for all sufficiently large $k$.
	\end{proof}
	\begin{cor}\label{nostrip}
		If there exists a centre branching Reeb component or tannulus, there exists a periodic centre annulus.
	\end{cor}
	Since Reeb components and tannuli are the only components of a rational-slope branching foliation which could have non-circle leaves, we also have the following.
	\begin{cor}\label{lem:circles}
		If $f$ does not admit any periodic centre annuli and $A$ has rational eigenvalues, the leaves of any invariant centre branching foliation are circles.
	\end{cor}
	The converse of this characterisation is also true---that a periodic centre annuli is necessarily a centre Reeb component or tannulus---though the ideas used to see this are precisely those we use to prove dynamical coherence in \cref{sec:coherent}. Since the converse is not needed for our results, we do not include a proof in this paper.
	
	\section{Degenerate linearisations}\label{sec:para}
	In this section, we prove \cref{thm:linearisation}. To relate the dynamics of the endomorphism to its linearisation, we introduce the notion of a ray associated to a branching foliation, which will only be used in this section.
	
    Let $f_0:\bbT^2\to\bbT^2$ be a partially hyperbolic endomorphism, and for the remainder of this section, suppose that $f_0$ does not admit a periodic centre annulus. Let $A$ be the linearisation of $f_0$. If $A$ has real eigenvalues, then by replacing $f_0$ by an iterate, we may assume that eigenvalues of $A$ are positive. Proving \cref{thm:linearisation} then amounts to ruling out the possibilities that $A$ has complex eigenvalues, is a homothety, or admits a non-trivial Jordan block, where in the latter two cases we may assume the eigenvalues of $A$ are positive.
	
	Lift $f_0$ to a diffeomorphism $f:\bbR^2\to\bbR^2$, and let $A$ also denote the lift of the linearisation. Further, lift the centre direction and unstable cone of $f_0$ to $\bbR^2$, and denote them $E^c$ and $\Cone^u$. Up to taking finite covers, we can assume that $E^c$ and $\Cone^u$ are orientable, and by replacing $f$ by $f^2$, that $Df$ preserves these orientations. Similarly, we assume that $A$ is an orientation preserving linear map.
	
	Recall that there exists an invariant centre branching foliation of $\bbR^2$ which descends to $\bbT^2$, as is constructed for surfaces in Section 5 of \cite{BI}. Moreover, given small $\eps>0$, the leaves of this branching foliation lie less than $\eps$ in $C^1$-distance from an approximating foliation $\feps$ which also descends to $\bbT^2$. This implies that $\feps$ lies a finite distance from a linear foliation of $\bbR^2$, and thus, so does $\bran$. Moreover, from \cref{lem:circles}, the absence of periodic centre annuli implies that $\feps$ descends to the suspension of a circle homeomorphism. This is already enough to rule out the case of complex eigenvalues.
	\begin{lem}\label{lem:complex}
		The linearisation $A$ of $f$ has real eigenvalues.
	\end{lem}
	\begin{proof}
		An invariant centre branching foliation $\bran$ which descends to $\bbT^2$ is a finite distance from a linear foliation of $\bbR^2$. Since $f$ and $A$ are a finite distance apart, this linear foliation must be $A$-invariant. A linear map of $\bbR^2$ with complex eigenvalues preserves no such foliation.
	\end{proof}
	
	Now we begin to consider to the homothety and Jordan block cases.
	\begin{lem}
	The linearisation $A$ must have an eigenvalue greater than $1$.
    \end{lem}
    \begin{proof}
    Let $U\subset \bbR^2$ be a small neighbourhood open neighbourhood. If the $A$ has no eigenvalue greater than $1$, then the volume of $A^{-n}(U)$ is bounded. Since $f$ is a finite distance from $A$, the volume of $f^{-n}(U)$ can grow at most polynomially. But $U$ contains some small unstable curve $J^u$, and so by \cref{lengthvolume} implies that the volume of $f^{-n}(U)$ must grow exponentially, a contradiction.
    \end{proof}

	Using the orientation on $E^c$, we can talk about forward centre curves emanating from a point. Given a branching foliation $\bran$ and a leaf $\mathcal{L}\in\bran$ through the origin, we call the \emph{forward half of $\mathcal{L}$} to be all points on $\mathcal{L}$ forward of the origin. Recall that a ray is a subset of $\bbR^2$ given by $\{t\cdot v\in\bbR^2: t\geq 0\}$ for some non-zero vector $v\in \bbR^2$. 
	\begin{lem}\label{lem:ray}
		Let $\bran$ be an invariant branching foliation. Then there exists a unique ray $R(\bran)$ emanating from the origin such that if $\mathcal{L} \in \bran$ passes through the origin, the forward half of $\mathcal{L}$ lies a finite distance from $R(\bran)$.
	\end{lem}
	\begin{proof}
		Let $\feps$ be an approximating foliation of $\bran$. By taking $T\feps$ to be transverse to $\Cone^u$, $T\feps$ has a natural orientation that is consistent with with that of $E^c$. Since $\feps$ descends to $\bbT^2$, it lies a finite distance from a linear foliation of $\bbR^2$. Thus the forward half of the leaf $\feps(p)$ lies close to some ray $R(\feps)$ emanating from the origin. Let $\mathcal{L}\in \feps$ and $J\subset \mathcal{L}$ be a subcurve obtained by taking all points forward of some point $q\in \mathcal{L}$. Since we know $\feps$ in fact descends to a suspension, then $J$ lies a finite distance from $R(\feps)$.
		
		Now the forward half of a leaf of $\bran$ is a finite distance from some forward segment $J$ contained in a leaf $\mathcal{L}\in \feps$ as described in the preceding paragraph. Thus any forward half a leaf of $\bran$ also lies forward half of $\mathcal{L}$ is a finite distance from $R(\bran):=R(\feps)$.
	\end{proof}

	Recall that there can be many invariant centre branching foliations, and by using the explicit constructions of \cite{BI}, we have two concrete examples as follows. Using the orientations on $\Cone^u$ and $E^c$, we may roughly view $\Cone^u$ locally as vertical and $E^c$ horizontal. One can obtain a branching foliation by taking the maximal highest forward centre curve through a point $q\in \bbR^2$ and gluing it to the lowest backward centre curve through $q$. We denote this branching foliation as $\fmax$. A second (though not necessarily distinct) branching foliation can be obtained by flipping this construction, and gluing the lowest forward to the highest backward curves. We denote this branching foliation $\fmin$.
	
	Consider now the space $S$ of rays. Since each ray may be assocated to a unit vector in $\bbR^2$, the space $S$ is homeomorphic to a circle. We specifically work with $S$ rather than projective space $\bbR P^1$ in order to handle the Jordan block case. Observe that a line on $\bbR^2$ quotients down to a pair of antipodal points on $S$, so that the complement of this line in $S$ has two connected components. 
	
	\begin{lem}\label{lem:split}
		There exists a line $\gamma \subset \bbR^2$ such that one connected component $S^+$ of $S\setminus \gamma$ contains $R(\bran)$ for any centre branching foliation. In particular, this component contains both $R(\fmax)$ and $R(\fmin)$. 
	\end{lem}
	\begin{proof}
	Since an invariant centre branching foliation is approximated by a suspension, there there exists a circle $C_0\subset \bbT^2$ passing through $(0,0)\in\bbT^2$ which is transverse to the centre direction. By redefining the unstable cone sufficiently large, we can further assume that $C_0$ is tangent to $\Cone^u$. The lift $C\subset\bbR^2$ of $C_0$ passing through the origin in $\bbR^2$ then lies close to a line $\gamma$. The line $\gamma$ divides the ray space $S$ into two connected components $S^-$ and $S^+$, and the curve $C$ divides $\bbR^2$ into two half spaces $\Gamma^-$ and $\Gamma^+$.
	
        If a leaf of the lifted center branching foliation intersected
        the lifted circle $C$, then by transversality
        a leaf of a sufficient close approximating foliation
        would also intersect twice.
        However, a Poincar\'e-Bendixson argument shows this is not possible
        (see \cref{uniquehit} in the next section for a similar argument).
        Thus, as a forward centre curve through the origin intersects $C$ at the origin, any such curve must lie entirely in either $C^-$ or $C^+$. The orientations on the transverse $E^c$ and $\Cone^u$ then imply then that all forward centre curves from the origin all in fact lie in all must lie on only one of these half spaces, which we take to be $C^+$. Now if a ray lies a finite distance from a forward centre curve, then it necessarily lies in $\Gamma^+$, or is one of the two points corresponding to $\gamma$ itself. However, every forward centre curve lies close to a curve obtained by lifting a circle in a suspension on $C_0$, so such a ray cannot be one of the points corresponding to $\gamma$.
    \end{proof}

	By giving an orientation on $S$ we obtain the usual notion of `clockwise' and `counterclockwise' on the circle. We assume without loss generality that the path from $R(\fmax)$ to $R(\fmin)$ that is contained in $S^+\subset S$ points in the clockwise direction. From a natural perspective this means that the ray $R(\fmin)$ lies clockwise from $R(\fmax)$.
	\begin{lem}\label{lem:perturbation}
		There are smooth foliations $\fcw$ and $\fccw$ of $\bbR^2$ which descend to $\bbT^2$ such that:
		\begin{enumerate}
		    \item both $T\fcw$ and $T\fccw$ are a small distance from the centre direction, and are transverse to $\Cone^u$
		    \item with respect to the orientation given by $\Cone^u$, $T\fcw$ lies below $E^c$ and $T\fccw$ above $E^c$
		    \item the rays $R(\fcw)$ and $R(\fccw)$ are contained in $S^+\subset S$
		    \item the path from $R(\fmin)$ to $R(\fcw)$ that is contained in $S^+$ goes in the clockwise direction, while the path from $R(\fmax)$ to $R(\fccw)$ that is contained in $S^+$ goes in the counterclockwise direction.
		    \item if $A$ has distinct eigenvalues, then $R(\fcw)$ and $R(\fccw)$ are not tangent to the $A$-invariant directions.
		\end{enumerate}
	\end{lem}
	\begin{proof}
		We shall start by finding $\fcw$. Let $\feps$ be an approximating foliation associated to $\bran$. Then $\feps$ descends to a suspension of homeomorphism on some circle $C_0\subset\bbT^2$  which is transverse to $E^c$. Now take a smooth distribution $E^{\mathrm{cw}}$ to lie sufficiently close to $E^c$ so that it remains transverse to the unstable cone, but that with respect to the orientation on the unstable cone, lies strictly below $E^c$.
        
        For concreteness, this construction may be done as follows.
        Let $X^c$ be the (continuous) vector field consisting of unit vectors that point in in the
        direction of $E^c$ with the same orientation as $E^c$,
        and let $X^u$ be any vector field consisting
        of unit vectors lying inside of the cone family $\mathcal{C}^u$
        and with the same orientation as the cone family.
        If $\eps > 0$ is small, then the vector field
        defined by $X = X^c - \eps X^u$
        consists of vectors close to $E^c$ but rotated slightly clockwise
        away from $\mathcal{C}^u$.
        This vector field is continuous and not smooth in general.
        By a small perturbation (much smaller than $\eps$),
        we can approximate $X$ by a smooth vector field $X^{\mathrm{cw}}$
        such that its vectors are close to $E^c$ but rotated slightly
        clockwise away from $\mathcal{C}^u$.
        This smooth vector field then defines $E^{\mathrm{cw}}$.

        As a smooth distribution, $E^{\mathrm{cw}}$ integrates to a smooth foliation $\fcw$. The return map of the flow along $T\fcw$ to the circle $C_0$ must then have a lower rotation number than that of the return map of $T\feps$, implying that $R(\fcw)$ lies a small clockwise distance in $S$ from $R(\fmin)$. With this, we obtain the third and fourth properties. The last property can be satisfied by refining our perturbation: If $R(\fcw)$ happened to point in the expanding $A$ direction, then another small clockwise perturbation of $T\fcw$ will integrate to the desired foliation.
		
		By instead perturbing $E^c$ to $E^{\mathrm{ccw}}$ lying strictly above $E^c$, the resulting distribution integrates to the desired $\fccw$ by a similar argument.
	\end{proof}
	The foliations $\fcw$ and $\fccw$ are thought of as clockwise and counterclockwise perturbations to the centre direction, respectively. To illustrate to the reader the objects we are considering in $S$, below is a figure of the ray space $S$ indicating the relationships we have established from the rays of $\fmax$, $\fmin$, $\fcw$ and $\fccw$. 
    
    \begin{figure}[h]
	    \includegraphics{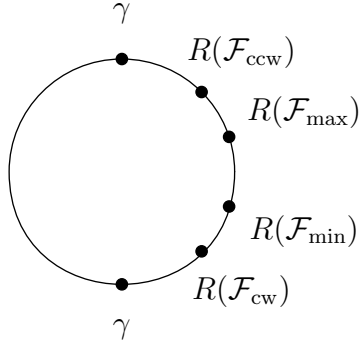}
		\caption{A depiction of the ray space $S$ and the points we are considering in S.}
		\label{fig:rayhelp}
	\end{figure}

	The next step is to iterate these perturbed foliations backward. Given any foliation $\mathcal{F}$ of $\bbR^2$ which descends to $\bbT^2$, define a foliation $f^{-n}(\mathcal{F})$ of $\bbR^2$ by $f^{-n}(\mathcal{F})(p) = f^{-n}(\mathcal{F}(f^n(p)))$. This sequence of foliations $\mathcal{F}_n$ does not necessarily descend to a foliation on $\bbT^2$, but we now show that it still has an associated ray. Observe that since the linearisation $A$ is an orientation preserving linear map, it induces an orientation preserving map $S\to S$, which we also denote as $A$.
	
	\begin{lem}\label{lem:iterate}
		Let $\mathcal{F}$ be either $\mathcal{F}_{cw}$ or $\mathcal{F}_{ccw}$, and define $\mathcal{L}\in f^{-n}(\mathcal{F})$ to be the leaf of $f^{-n}(\mathcal{F})$ through the origin. Then the forward half of $\mathcal{L}$ lies a bounded distance from the ray $R(f^{-n}(\mathcal{F})):= A^{-n}(R(\mathcal{F}))$. This distance is uniformly bounded for all $n$ and choices of $\mathcal{L}$.
	\end{lem}
	\begin{proof}
		If $L\subset\bbR^2$ is the line containing $R(\mathcal{F})$, then using the fact that $A$ is a finite distance from $f$, it is straightforward to see that any leaf of $f^{-n}(\mathcal{F})$ is a bounded distance from $A^{-n}(R(\mathcal{F}))$. This distance is naturally bounded across all choices of $\mathcal{L}$.
		To see uniform boundedness for all $n$, we refer the reader to \S 2 of \cite{enco}. The proof of this property there again uses the finite distance between $f$ and $A$ to show the result under the assumption that $A$ is hyperbolic and $R(\mathcal{F})$ is not tangent to the contracting eigendirection of $A$. When $A$ has an eigenvalue of $1$, then the proof is nearly identical, and uses the final property of \cref{lem:perturbation}. When $A$ is expanding, the argument is easier.
	\end{proof}	
	Suppose further now that $T\mathcal{F}$ is transverse to $E^c$ and $\Cone^u$, as is the case for $\fcw$ and $\fccw$. Then the tangent lines $Tf^{-n}(\mathcal{F})$ converge to the centre direction as $n\to\infty$. Since $f$ is partially hyperbolic there exists a centre cone family, that is, a cone family transverse to $\Cone^u$ which contains both $T\mathcal{F}$ and $E^c$, and that is invariant under $Df^{-1}$. This cone family inherits a $Df^{-1}$-invariant orientation from the orientation on $E^c$. Using this orientation, we consider the sequence of curves $\mathcal{L}_n$ emanating from the origin given by the forward half of the leaf of $f^{-n}(\mathcal{F})$ through the origin. By an Arzela-Ascoli argument, this sequence has a convergent subsequence in the compact open topology. The limit of such a subsequence is then a complete forward centre curve emanating from the origin. We use this to prove the following.
	
	\begin{lem}\label{lem:limit}
		We have $lim_{n\to\infty}(A^{-n}(R(\fcw))) = R(\fmin)$ and\newline $lim_{n\to\infty}(A^{-n}(R(\fccw))) = R(\fmax)$.
	\end{lem}
	\begin{proof}
		We prove the first limit, the second follows by similar argument. Let $\mathcal{L}_n$ be the forward half of the leaf in $f^{-n}(\fcw)$ that passes through the origin. By the discussion preceding the lemma, the sequence $\mathcal{L}_n$ has a convergent subsequence, and the limit $\mathcal{L}$ of this subsequence is a complete forward centre curve through the origin. By \cref{lem:iterate}, the limit $\mathcal{L}$ lies a uniformly bounded distance from $\lim_{n\to\infty}A^{-n}(R(\fcw))$, which we note does exist since $A$ cannot have complex eigenvalues. Thus, if we can show that $\mathcal{L}$ is actually the forward half of a leaf in $\fmin$, we will have proved the result.
		
		Recall that $T\fcw$ lies strictly below $E^c$ with respect to the orientation on $\Cone^u$. Since this orientation is preserved by $Df$, then $T f^{-n}(\mathcal{F})$ lies strictly below $E^c$ for all $n$. Thus the limit $\mathcal{L}$ must in fact be the lowest forward centre curve through the origin. This is, by definition, the forward half of a leaf of $\fmin$, giving the result.
	\end{proof}
		The final observation we need to make is that the dynamics of the map $A:S\to S$ is determined by the eigenvalues of the linearisation. An invariant line of $A$ quotients down to an antipodal pair of fixed points in $S$. If $A$ has distinct eigenvalues, the eigenline corresponding to the greater eigenvalue corresponds a pair of attracting fixed points, while the other eigenline becomes two repelling fixed points. If $A$ is a homothety, the induced map on $S$ is the identity. When $A$ has a non-trivial Jordan block, the map has just one pair of fixed points which are each attracting on one side and repelling on the other. Each of these behaviours is demonstrated in \cref{fig:rays}, and they allow us rule out the remaining cases.
	\begin{figure}[h]
		\includegraphics{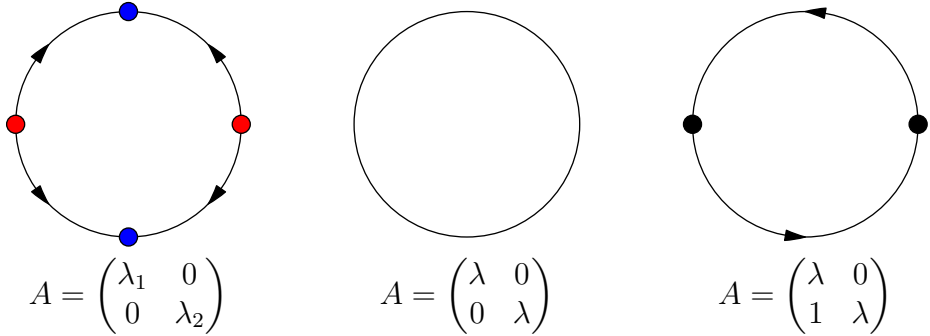}
		\caption{The dynamics of the map $A:S\to S$. From left to right are the cases when $A$ has distinct eigenvalues, $A$ is a homothety, and when $A$ has a non-trivial Jordan block.}
		\label{fig:rays}
	\end{figure}
	
	\begin{proof}[Proof of \cref{thm:linearisation}]
		The possibility of $A$ admitting complex eigenvalues was ruled out in \cref{lem:complex}.
		
		Next suppose that $A$ is a homothety, so that its induced map $S\to S$ is the identity. Thus $A^{-n}(R(\fcw)) = R(\fcw) \neq R(\fmin)$, which contradicts the fact that $\lim_{n\to\infty}(R(\fcw))= R(\fmin)$ by \cref{lem:limit}.

		Finally, suppose that $A$ has a non-trivial Jordan block. Then the map $A:S\to S$ has only two fixed points, each of which is repelling on one side and attracting on the other, as illustrated in \cref{fig:rays}. Only one of these points can lie in the component $S^+$ of $S\setminus \gamma$, and so $R(\fmax)$ and $R(\fmin)$ must both be this fixed point. Then one of $R(\fcw)$ and $R(\fccw)$ lies on the side of this fixed point that is attracting under $A$, and the other, the side that is repelling. Without loss of generality we may assume then that $A^{-n}(R(\fcw))$ converges to $R(\fmax) = R(\fmin)$, but that $A^{-n}((R(\fccw)))$ converges to the point antipodal to $R(\fmin)$. This once again contradicts \cref{lem:limit}.
	\end{proof}

	With \cref{thm:linearisation} proved, we conclude the section by using the tools we have developed to establish a property that will be used in the construction of a leaf conjugacy. We relax our assumption that $A$ is expanding, so that we are assuming that $A$ is any linearisation $f$ of a partially hyperbolic endomorphism without periodic centre annuli. We have shown that $A$ must have eigenvalues of distinct magnitude, so let $\mathcal{A}^c$ denote the linear foliation of $\bbT^2$ associated to the smaller eigenvalue of $A$.
	\begin{lem}\label{lem:almostparallel}
		Any invariant centre branching foliation $\bran$ is a bounded distance from $\mathcal{A}^c$ on the universal cover.
	\end{lem}
	\begin{proof}
	    We have shown that $A$ must have distinct eigenvalues. The map $A:S\to S$ then has two pairs of fixed points, each pair corresponding to the eigenvalues $\lambda^c$ and $\lambda^u$ of $A$, with $\lambda^c<\lambda^u$. By the last item of \cref{lem:perturbation}, the rays $R(\fcw)$ and $R(\fccw)$ are not any of the fixed points of $A:S\to S$. Then $R(\fmax) = \lim_{n\to\infty}(A^{-n}(R(\fcw)))$ and $R(\fmin) = \lim_{n\to\infty}(A^{-n}(R(\fccw)))$ are both fixed points associated to $\lambda^c$. Only one such fixed point lies in $S^+$, so $R(\fmax)=R(\fmin)$. Since any forward centre curve must lie in between a cone region bounded be the lowest and highest forward centre curves, and this region must lie a finite distance from $R(\fmax)$, it follows that $R(\bran)=R(\fmax)$ for any invariant centre branching foliation $\bran$. Since $R(\fmax)$ is a fixed point associated to $\lambda^c$, then $\bran$ is a finite distance from $\mathcal{A}^c$.
	\end{proof}

	\section{Coherence and leaf conjugacy}\label{sec:coherent}
	In this section, we prove \cref{thm:classi}. Let $f_0:\bbT^2\to\bbT^2$ be partially hyperbolic. We assume for the remainder of this section $f_0$ does not admit any periodic centre annulus. Lift $f_0$ to $f:\bbR^2\to\bbR^2$, and similarly lift the linearisation of $f_0$ to $A:\bbR^2\to\bbR^2$. As an abuse of notation, we denote $\bran$ and $\feps$ as the lifts of the branching and approximating foliations respectively. Since $f_0$ does not admit any periodic centre annuli, then by \cref{thm:linearisation}, the linear map $A$ has distinct real eigenvalues $\lambda_1,\,\lambda_2$. Let $\mathcal{A}^c$ be the linear foliation of $\bbR^2$ associated to the smaller eigenvalue of $A$. Similarly let $\mathcal{A}^u$ be a linear foliation associated to $\lambda_2$. Define a projection $\pi^u:\bbR^2\to\bbR$ to be a linear map which maps each line of $\mathcal{A}^u$ onto $\bbR$ and whose kernel is the leaf of $\mathcal{A}^c$ through the origin.
    We will later use a projection $\pi^c : \bbR^2 \to \bbR$
    which is defined analogously.
    \begin{lem}\label{lem:Cbound}
		There is $C>0$ such that if $q\in\bran(p)$, then $|\pi^u(p)-\pi^u(q)|<C$.
	\end{lem}
	\begin{proof}
		The branching foliation $\bran$ is almost parallel with $\mathcal{A}^c$ by \cref{lem:almostparallel}, which implies that leaves of $\bran$ are a uniformly bounded distance in $\pi^u$ direction.
	\end{proof}
    Recall from the introduction that there exists an
    unstable foliation $\mathcal{F}^u$ on $\bbR^2$. A key step in proving that $\bran$ is in fact a true foliation will be showing that leaf segments of $\mathcal{F}^u$ grow large in the $\pi^u$-direction under forward iteration. When $A$ is hyperbolic or non-hyperbolic (cases defined in the introduction), this property of $\mathcal{F}^u$ can be proved by using the Poincar\'e-Bendixson Theorem and `length vs.~volume' arguments, as is done in Section 2 of \cite{enco}. When $A$ is expanding, it becomes difficult to compare growth rates of centre and unstable curves, so length vs.~volume arguments are less practical. Instead, we will use the Poincar\'e-Bendixson Theorem paired topological properties of $\feps$, which are now well understood due to the results in \cref{sec:inside}.
	\begin{prop}\label{uniquehit}
		If $\lf\in\feps$, then $\lf$ can intersect each leaf of $\mathcal{F}^u$ at most once.
	\end{prop}
	\begin{proof}
		Suppose that a curve in $\bran$ intersects a leaf of $\mathcal{F}^u$ more than once. Since this curve of $\bran$ is transverse to $\mathcal{F}^u$, then by the Poincar\'e-Bendixson Theorem, $\mathcal{F}^u$ must have a leaf which is a circle. As $\mathcal{F}^u$ is a foliation consisting of lines, this is a contradiction.
	\end{proof}
	Since $\bran$ contains no periodic centre annuli, then $\feps$ contains no Reeb components or tannuli, so $\feps$ is the lift of a suspension. This implies the following property.
	\begin{lem}\label{sensiblehomeo}
		The leaf space $\bbR^2/\feps$ of $\feps$ is homeomorphic to $\bbR$. Moreover, the corresponding homeomorphism can be chosen to satisfy the following property: There exists a deck transformation $\tau \in \bbZ^2$ such that if $\mathcal{L}^c$ corresponds to $l \in\bbR$ in the leaf space, then $\tau^n(\mathcal{L}^c)$ corresponds to $l + n$.
	\end{lem}
	We will frequently identify a leaf in $\feps$ with its corresponding representative in $\bbR$. If $X\subset\bbR^2$, define $U_{\eta}(X)=\{x\in\bbR^2:\dist(x,X)<\eta\}$.
	\begin{lem}\label{leavenbhd}
		If $\lf\in\feps$, then each of the connected components of $U_\eta(\lf)\setminus \lf$ contains a leaf of $\feps$.
	\end{lem}
	\begin{proof}
		The only foliations lifted from $\bbT^2$ which do not satisfy the desired property are those which admit Reeb components and tannuli. If $\bran$ contained either of these, $f$ would admit a periodic centre annulus.
	\end{proof}
	A \emph{$\Cone^u$-curve} is a curve tangent to $\Cone^u$; we will use $\Cone^u$-curves instead of unstable curves to utilise the compactness of $\bbT^2$. Since the unstable cone and centre direction are a bounded angle apart, then a $\Cone^u$-curve must `progress' some amount in the leaf space of $\feps$. This is captured in the following lemma.
	\begin{lem}\label{globaldelta}
		There is $\delta>0$ such that given a leaf $\mathcal{L}\in\feps$ which corresponds to $l \in \bbR$, the following property is satisfied: if $J^u$ is a $\Cone^u$ curve of length $1$ with an endpoint on $\mathcal{L}$, then the endpoints of $J^u$ lie on leaves which are at least $\delta$ apart in the leaf space of $\feps$.
	\end{lem}
	\begin{proof}
		We begin by establishing the property for a single leaf. Let $\eta>0$ be very small. Then since $E^c$ and $\Cone^u$ are a bounded angle away from each other, if $p\in\mathcal{L}$ and $J^u$ is a unit length $\Cone^u$ curve with $p$ as an endpoint, then $J^u$ is not contained within $U_{\eta}(\mathcal{L})$. By \cref{leavenbhd}, $J^u$ intersects a leaf $\hat{\lf}$. If $\hat{l}$ corresponds to $\lf$, then define $\delta_{l} = |\hat{l}-l|$.
		
		Now consider $[0,1]\subset\bbR$ in the leaf space. Since $[0,1]$ is compact, the existence of $\delta_l$ for each $l\in[0,1]$ implies that there is $\delta>0$ such that all leaves in $[0,1]$ satisfy the property in the proposition. Now suppose that $\mathcal{L}\in\mathcal{F}^c$ and that $J^u$ is a $\Cone^u$ curve of unit length with an endpoint on $\mathcal{L}$. If $\tau$ is as in \cref{sensiblehomeo}, then $\tau^n\lf\in[0,1]$ for some $n$, so the claim holds for $\tau^n\lf$. Since the unstable cone commutes with deck transformations, $\tau^n(J^u)$ is a $\Cone^u$ curve of unit length with an endpoint on $\tau^n\lf$, so the property holds for $\lf$.
	\end{proof}
	This progression of unstable curves in the leaf space of $\mathcal{F}^c_{\eps}$ then implies that large unstable curves will be long in the $\pi^u$-direction.
	\begin{lem}\label{unstableendpoints}
		If $q\in\mathcal{F}^u(p)$ and $q\ne p$, then $\sup_n|\pi^uf^n(p)-\pi^uf^n(q)|=\infty$.
	\end{lem}
	\begin{proof}
		Recall that the leaves of $\feps$ are uniformly bounded in the $\pi^u$-direction. This implies that for any constant $D>0$, there is $d>0$ such that if $p$ and $q$ are on leaves that are a distance $d$ apart in the leaf space, then $|\pi^u(p)-\pi^u(q)|>D$. Let $J^u$ be an unstable curve with endpoints $p$ and $q$. For large enough $n$, the iterate $f^n(J^u)$ has length greater than $d/\delta$. We can then decompose $f^n(J^u)$ into at least $d/\delta$ subcurves of unit length. Each subcurve, by \cref{globaldelta}, progresses $\delta$ in the leaf space of $\feps$. Since $f^n(J^u)$ cannot intersect a leaf of $\feps$ twice, then $f^n(p)$ and $f^n(q)$ must lie on leaves which are at least a distance $d$ apart in the leaf space. Hence $|\pi^uf^n(p)-\pi^uf^n(q)|>D$.
	\end{proof}
	Now given two leaves of $\bran$ which are connected by an unstable curve, the preceding lemma implies that these two leaves must grow apart in the $\pi^u$-direction under forward iteration by $f$. We use this argument to show that leaves of $\bran$ do not intersect, establishing coherence.
	\begin{prop}\label{coherence}
		Let $f_0:\bbT^2\to\bbT^2$ be a partially hyperbolic endomorphism. If $f_0$ does not admit a periodic centre annulus, then it is dynamically coherent.
	\end{prop}
	\begin{proof}
        Suppose there are two leaves $\lf_1,\, \lf_2 \in \bran$ which coincide at a point $p\in\bbR^2$. Consider the intervals $\pi^uf^n(\mathcal{L}_1)$ and $\pi^uf^n(\mathcal{L}_2)$. Since both of these intervals contain $\pi^uf^n(p)$, their union is an interval, \cref{lem:Cbound} implies that the length of this interval is at most $2C$. In particular, if $q_1\in\mathcal{L}_1$ and $q_2\in\mathcal{L}_2$ are points connected by an unstable segment, then $|\pi^uf^n(q_1)-\pi^uf^n(q_2)|$ is bounded. This contradicts \cref{unstableendpoints}.
		
		The branching foliation $\bran$ gives a partition of $\bbT^2$ tangent to $E^c$, and therefore is a foliation $\mathcal{F}^c$ by Remark 1.10 in \cite{BoWi}. This foliation is $f$-invariant. One can show that the foliation commutes with deck transformations, so $\mathcal{F}^c$ descends to the desired foliation on $\bbT^2$. Hence $f$ is dynamically coherent.
	\end{proof}
	Let $\mathcal{F}^c$ be the invariant centre foliation from the preceding proof.
	\begin{prop}
		The foliations $\mathcal{F}^c$ and $\mathcal{F}^u$ of $\bbR^2$ have global product structure.
	\end{prop}
	\begin{proof}
		\cref{uniquehit} gives uniqueness of an intersection between leaves of the foliations. Observe that since $\mathcal{F}^c$ has no Reeb components, then $\pi^c({\mathcal{L}^c})=\bbR$ for all $\lf^c\in\mathcal{F}^c$. One can use \cref{globaldelta} to show that if $\mathcal{L}^u\in\mathcal{F}^u$, then $\pi^u(\mathcal{L}^u)=\bbR$. By an Intermediate Value Theorem argument, $\lf^c$ must intersect $\lf^u$.
	\end{proof}
	With dynamical coherence established, we complete the proof of \cref{thm:classi} by constructing a leaf conjugacy from $f$ to $A$. When $f$ has a non-hyperbolic linearisation, then we do not necessarily have the Franks semiconjugacy. However, the leaf conjugacy can be obtained as an averaging of a map that is obtained in the construction of the semiconjugacy.
	\begin{prop}
		There exists a map $H^c:\bbR^2\to\mathcal{A}^c$ which sends a point $p\in\bbR^2$ to the unique line in $L\in\mathcal{A}^c$ which satisfies
		\[
		\sup_n \dist(f^n(p),A^n(L))<\infty.
		\]
		Moreover, $H^c$ commutes with deck transformations, and the induced map $\bbR^2\to\bbR^2/\mathcal{A}^c$ is continuous.
	\end{prop}
	\begin{proof}
		This map is constructed in the construction of the Franks semiconjugacy in \cite{Franks1}. The proof relies only on the fact that $f$ and $A$ induce the same homomorphism on $\pi_1(\bbT^2)$, and that the greatest eigenvalue of $A$ is greater than $1$. This is the case for all partially hyperbolic surface endomorphisms. For a sketch of the argument, see Theorem 3.1 of \cite{hp-survey}.
	\end{proof}
	Now we establish properties of $H^c$ to show it is suitable for using an averaging technique.
	\begin{lem}\label{bijection}
		We have $H^c(q)=H^c(p)$ if and only if $q\in \mathcal{F}^c(p)$.
	\end{lem}
	\begin{proof}
		First, suppose that $q\in\mathcal{F}^c(p)$. Then $|\pi^uf^n(p)-\pi^uf^n(q)|<C$. Then if $L\in\mathcal{A}^c$  is such that $\dist(f^n(p),A^n(L))$ is bounded, then as is $\dist(f^n(q),A^n(L))$. Hence $H^c(p)=H^c(q)$.
		
		Now suppose that $H^c(q)=H^c(p)$ but that $q\notin \mathcal{F}^c(p)$. Then $f^n(p)$ and $f^n(q)$ must both lie close to $A^n(L)$ for some $L\in\mathcal{A}^c$. This implies that $|\pi^uf^n(p)-\pi^uf^n(q)|$ is bounded. However, by global product structure, there exists $\hat{q}\in\mathcal{F}^u(p)\cap\mathcal{F}^c(q)$. Since centre leaves are bounded in the $\pi^c$-direction, then $|\pi^uf^n(\hat{q})-\pi^uf^n(q)|$ must bounded, so that $|\pi^uf^n(p)-\pi^uf^n(\hat{q})|$ is also bounded. This contradicts that $\hat{q}\in\mathcal{F}^u(p)$. Hence $q\in\mathcal{F}^c(p)$.
	\end{proof}
	For $q\in\mathcal{F}^c(p)$, let $d_c(p,q)$ be the distance of the leaf segment from $p$ to $q$.
	\begin{lem}\label{length}
		There exists $T>0$ such that if $q\in\mathcal{F}^c(p)$ and $d_c(p,q)>T$, then $|\pi^c(p)-\pi^c(q)|>1$.
	\end{lem}
	\begin{proof}
		This follows immediately from the fact that $\mathcal{F}^c$ is equivalent to the suspension of a circle homeomorphism that is almost parallel to $\mathcal{A}^c$. 
	\end{proof}
	Finally, we construct the leaf conjugacy.
	\begin{proof}[Proof of \cref{thm:classi}]
		Fix $\mathcal{L}\in\mathcal{F}^c$ and let $\alpha: \bbR\to\bbR^2$ be an arc length parametrisation of $\mathcal{L}$. For $p \in \mathcal{L}$, let $s = \alpha^{-1}(p)$. Let $T$ be as in \cref{length} and define $h(p)$ as the unique point in $H^c(\mathcal{L})$ which satisfies
		\[
		\pi^c h(p) = \frac{1}{T}\int_0^T \pi^c \alpha(s+t)\, dt.
		\]
		
		Define $h$ on each leaf of $\mathcal{F}^c$ to obtain a map $h:\bbR^2\to\bbR^2$. This map $h$ is homeomorphism, and descends to the desired leaf conjugacy on $\bbT^2$. The details of this argument are identical to that of Section 3 and the proof of Theorem B in \cite{enco}, where the map $H^c$ takes the place of the Franks semiconjugacy $H$, with Proposition 3.2 and Lemma 3.3 replaced by \cref{bijection} and \cref{length} of the current paper.
	\end{proof}

	\section{Maps with periodic centre annuli}\label{sec:orbits}
	With the preceding section having proven a classification in the absence of a periodic centre annulus, we complete the classification by addressing those with such annuli. That is, we prove Theorems \ref{thm:nonhypclassi} and \ref{thm:expandclassi}. Let $f:\bbT^2\to\bbT^2$ be partially hyperbolic surface endomorphism which admits a periodic centre annulus. 
	
	Central to our classification is the following proposition, which states that the dynamics on a periodic centre annulus take the form of a skew product.
	
	\begin{prop}\label{prop:skew}
		Suppose $f$ is a partially hyperbolic endomorphism of a closed, oriented
		surface $M$ 
		and that there is an invariant annulus $M_0$ with the following
		properties{:}
		
		\begin{enumerate}
			\item
			$f(M_0) = M_0$ and $f$ restricted to $M_0$ is a covering map;
			
			\item
			the boundary components of $M_0$ are circles tangent to the
			center direction.
			
			\item
			no circle tangent to the center direction intersects the interior
			$U_0$ of
			$M_0$.
		\end{enumerate}
		Then,
		there is an embedding $h: U_0 \to S^1 \times \bbR$
		such that the homeomorphism
		$h \circ f^k \circ h \inv$ from $h(U_0)$ to itself
		is of the form
		\[        h f^k h \inv(v, s) = (A(v), \phi(v,s) )  \]
		where A : $S^1  \to  S^1$ is an expanding linear map
		and $\phi:h(U_0) \to \bbR$ is continuous.
		Moreover,
		if $v  \in  \bbT^2$, then 
		$h \inv(v \times \bbR)$
		is a curve tangent to $E^c$.
	\end{prop}
	This result is an analogue of the classification of dynamics of diffeomorphisms with centre-stable tori. In fact, the preceding theorem may be proved by adjusting each of the arguments used in \cite{clab} to the current setting, where one replaces the notion of a region between centre-stable tori with a periodic centre annulus, both of which take the role of $M_0$ in the statement above. Since this is long process which does not use any new ideas, it is completed in an auxiliary document, \cite{clab2}.
	
	With the dynamics inside a periodic centre annulus understood, we turn to outside of the annulus, and in particular the preimages of these annuli. Let $\bran$ be an invariant branching foliation which contains all the periodic centre annuli, as in \cref{lem:inbran}. Let $A:\bbT^2\to\bbT^2$ be the linearisation of $f$. Then the lift of $\bran$ to $\bbR^2$ lies close to the lift of some $A$-invariant linear foliation $\mathcal{A}$. Let $\lambda_1$ be the eigenvalue of $A$ associated to the foliation $\mathcal{A}$, and $\lambda_2$ the other (not necessarily distinct) eigenvalue of $A$.
	\begin{lem}\label{lem:circles}
		Let $\mathcal{S}\in \bran$ be a centre circle. Then the preimage $f^{-1}(\mathcal{S})$ consists of $\lambda_2$ disjoint circles in $\bran$.
	\end{lem}
	\begin{proof}
		Since $f$ is a covering map, $f^{-1}(\mathcal{S})$ consists of circles, each of which is mapped onto $\mathcal{S}$. These circles cannot be null-homotopic since they are tangent to the centre.  Let $\mathcal{S}'\subset f^{-1}(\mathcal{S})$. Then since $f(\mathcal{S}')=\mathcal{S}$ has the same slope as $\mathcal{A}$ and $f$ is a finite distance from $A$, then $\mathcal{S}'$ is also has the same slope as $\mathcal{A}$. Since the induced homomorphism of $f$ on the fundamental group of $\bbT^2$ is the same as that of $A$, then $f\vert_{\mathcal{S}'}:\mathcal{S}'\to\mathcal{S}$ is a map of degree $\lambda_1$. Then as $f$ has degree $\lambda_1\lambda_2$, the set $f^{-1}(\mathcal{S})$ must consist of $\lambda_2$ distinct circles.
	\end{proof}
	The remaining ingredient for our two classification theorems is that outside the orbits of periodic centre annuli are circles in $\bran$.
	\begin{prop}\label{lem:dense}
		Suppose that $f$ admits a periodic centre annulus. Let $X_1,\dots,X_n$ be the distinct, disjoint periodic centre annuli of $f$. The connected components of the set
		\[
		V = \bbT\setminus \bigcup_i\bigcup_{k\in\bbZ}f^k(X_i)
		\]
		are circles.
	\end{prop}
	\begin{proof}
		Let $\bran$ be a branching foliation which contains the leaves that saturate all periodic centre annuli and their boundaries. As $V$ is the complement of the preimages of these annuli, then $V$ is also saturated by leaves of $\bran$. A property of foliations on $\bbT^2$ is that if a leaf is not a circle, then it lies inside a tannulus or Reeb component. By \cref{sec:inside}, the non-circle leaves then lie only in periodic centre annuli or their preimages. Moreover, by a Reeb compact leaf argument, the limits of  accumulating boundary circles in the preimages of periodic centre annuli is also a circle. This ensures that $V$ will be both a set saturated by leaves, and these leaves must be circles.
		
		Let $Y \subset V$ be a connected component, and suppose that $Y$ has non-empty interior. Since $Y$ is saturated by leaves of $\bran$, then $Y$ can a priori be either an annulus, or a `pinched annulus', i.e., a deformation retract of an annulus.  Let the width of $Y$ be the Hausdorff distance between its boundary circles, as we used in \cref{sec:inside}. Then arguing similarly to \cref{stripsareperiodic}, for all sufficiently large $k$, $f^k(Y)$ has width bounded from below. Then $f^n(Y)$ must be pre-periodic, so that $V$ contains a periodic centre annulus, contradicting the definition of $V$.
	\end{proof}
	We consider now the setting of a non-hyperbolic linearisation and aim to prove \cref{thm:nonhypclassi}. This relies primarily on a length vs.~volume argument.
	\begin{lem}\label{lem:nonhyp_eig}
		If $A$ is nonhyperbolic, a periodic centre annulus has the slope associated to the eigenvalue $\lambda$ of $A$ for which $\lambda>1$.
	\end{lem}
	\begin{proof}
		Let $X_0\subset\bbT^2$ be a periodic centre annulus. Assume without loss of generality that $X_0$ is invariant, and lift $X_0$ to a strip $X\subset \bbR^2$. Assume that the claim is false, so that $X_0$ has associated eigenvalue $\lambda=1$. By making an affine change of coordinates on $\bbR^2$, we can assume that the linearisation takes the form $A: (x,y)\mapsto (\lambda x,y)$ for the eigenvalue $\lambda>0$ of $A$, and that $X \subset \bbR \times [0,1]$. Let $U = [-1,1]\times [0,1]\subset \tilde{X}$. Lift $f$ to a diffeomorphism $\tilde{f}:\bbR^2\to\bbR^2$ that fixes $X$.
		\begin{sublem}
			There is $D>0$ such that the set $U$ satisfies
			\[
			\tilde{f}^n(U)\subset [-Dn,Dn]\times [0,1].
			\]
		\end{sublem}
		\begin{proof}
			If $C>0$ is the $C^{0}$ distance between $A$ and $f$, then $A^{n}([-1,1]\times\bbR)\subset [-1,1]\times\bbR$, so that $\tilde{f}^{n}([-1,1]\times\bbR)\subset [-1-Cn,1+Cn]\times\bbR$, and so if $D=1$, then $\tilde{f}^{n}([-1,1]\times\bbR)\subset [-Dn,Dn]\times\bbR$. Since $\bbR\times[0,1]$ is $\tilde{f}$ invariant, then $\tilde{f}^{n}([-1,1]\times\bbR)\subset [-Dn,Dn]\times[0,1]$.
		\end{proof}
		Now if $J^u\subset U$ is a small unstable curve unstable curve, then by \cref{lengthvolume}, the volume of $\tilde{f}^n(U)$ must grow exponentially in $n$. However, the sublemma above shows that the volume of $\tilde{f}^n(U)$ can grow at most polynomially, giving a contradiction.
	\end{proof}
	This allows us to conclude the classification in the case of a non-hyperbolic linearisation.
	\begin{proof}[Proof of \cref{thm:nonhypclassi}]
		By \cref{lem:circles} and \cref{lem:nonhyp_eig}, the preimage of the boundary circle of a periodic centre annulus is a single circle. In turn, the preimage of a single periodic centre annulus is a single annulus, which by finiteness of the periodic centre must also be periodic. By \cref{lem:dense}, the finitely many periodic centre annuli together with their boundary circles cover $\bbT^2$.
	\end{proof}
	Finally, we complete the classification in the case of an expanding linearisation.
	\begin{proof}[Proof of \cref{thm:expandclassi}]
		Suppose that $A$ is expanding. Then as both eigenvalues of $A$ are larger than $1$, \cref{lem:circles} shows that the preimage of a periodic centre annulus is necessarily multiple disjoint annuli. By continually iterating backwards, we see that the orbit of a periodic centre annulus consists of infinitely many annuli, and their union is dense by \cref{lem:dense}.
		Observe that points within a given periodic centre annulus remain a bounded distance apart when iterated on the universal cover, implying that the image of such an annulus under the semiconjugacy $H$ is necessarily a circle. By the semiconjugacy property, this circle will be either preperiodic or periodic under $A$. Similarly, a circle $S$ which is not in the orbit of the periodic centre annuli will be mapped to circles by $H$. If $S$ is not preperiodic or periodic under $f$, then $H(S)$ is necessarily transitive under $A$.
	\end{proof}
	
	\printbibliography
\end{document}